\documentclass[12pt,letterpaper]{amsart}
\usepackage{latexsym,amssymb,amsmath,amsthm,mathtools,bm}
\usepackage{seqsplit}
\usepackage{url}

\def\C{\mathbb C}
\def\R{\mathbb R}

\def\E{\mathbb E}

\def\HH{{\bf H}}

\newtheorem{theorem}{Theorem}[section]

\newtheorem{proposition}[theorem]{Proposition}

\theoremstyle{definition}

\newtheorem{claim}{Claim}

\makeatletter
    
    \@addtoreset{equation}{section}
  \makeatother

\begin{document}

\title
[Biharmonic rotational surfaces in $\E^4$]
{Biharmonic rotational surfaces in the four-dimensional Euclidean space are minimal}

\author{Shun Maeta}
\address{Department of Mathematics, Chiba University, 1-33, Yayoicho, Inage, Chiba, 263-8522, Japan.}
\curraddr{}
\email{shun.maeta@faculty.gs.chiba-u.jp~{\em or}~shun.maeta@gmail.com}

\date{}

\dedicatory{}


\commby{}

\begin{abstract}
In this paper, we show that any biharmonic simple rotational surface in the four-dimensional Euclidean space is minimal. The proof is based on reducing the biharmonic equation to a system of ordinary differential equations for the profile curve and then excluding all possible non-minimal branches. This is a partial affirmative answer to Chen's conjecture. 
\end{abstract}

\maketitle


\section{Introduction}\label{intro}

In 1988, Bang-Yen Chen proposed a conjecture on biharmonic submanifolds in a Euclidean space, that is, {\bf any biharmonic submanifold $(\Delta {\bf H}=0)$ in a Euclidean space is minimal (${\bf H}=0$)} (cf. \cite{C}). 
The study of Chen's conjecture has focused mainly on the case of hypersurfaces.
Chen and Jiang independently proved that biharmonic surfaces in $\E^3$ are minimal \cite{C,J2}. 
Hasanis and Vlachos \cite{HV} and Defever \cite{D} proved that biharmonic hypersurfaces in $\E^4$ are minimal. 
Before the recent breakthrough of Fu, Hong, and Zhan, several affirmative results had also been obtained under additional geometric assumptions. 
For instance, Luo proved the conjecture for weakly convex hypersurfaces \cite{L}. 
Fu and Hong proved the conjecture for hypersurfaces with constant scalar curvature
and at most six distinct principal curvatures \cite{FH}.
Montaldo, Oniciuc, and Ratto verified the conjecture for cohomogeneity one hypersurfaces \cite{MOR}. 
Deepika and Arvanitoyeorgos proved it for $\delta(r)$-ideal hypersurfaces \cite{DA}.

Recently, significant progress has been made on Chen's conjecture for hypersurfaces.
Fu, Hong, and Zhan proved that every biharmonic hypersurface in $\E^5$ is minimal \cite{FZ1}. 
They subsequently proved Chen's conjecture for biharmonic hypersurfaces in $\E^6$ \cite{FZ2}. 
In addition to these elliptic approaches, Fu, Hong, and Tian have recently initiated a flow-theoretic approach to biharmonic hypersurfaces by studying the biharmonic hypersurface flow in higher dimensions \cite{FHT}.

On the other hand, much less is known in higher codimension. 
Indeed, even the case of surfaces in $\E^4$ has not yet been completely solved. 
For proper immersions, Akutagawa and the author proved that every biharmonic properly immersed submanifold in a Euclidean space is minimal \cite{AM}.
Thus, under the properness assumption, Chen's conjecture is already known in all codimensions.
Consequently, in order to obtain a result which is not covered by this global theorem, it is important to establish a local non-existence result without assuming properness of the immersion.
We remark that, for an isometric immersion into Euclidean space, properness of
the immersion implies completeness with respect to the induced metric.
Hence, it is important to study biharmonic surfaces in $\E^4$ from a local point of view.
Some partial results are known under additional assumptions on the mean curvature vector. 
The author and Urakawa gave an affirmative answer to Chen's conjecture for Lagrangian surfaces with parallel normalized mean curvature vector in $\C^2=\E^4$.
Ye\u{g}in \c{S}en and Turgay proved that there is no proper biharmonic surface with parallel normalized mean curvature vector in $\E^4$ \cite{ST}.
This result was later generalized by Chen, who proved that biharmonic surfaces with parallel normalized mean curvature vector do not exist in any Euclidean space \cite{ChenPNMCV}.

It is also useful to compare the biharmonic condition with the biconservative condition.
In the Euclidean setting, the biconservative condition is given by
$(\Delta \HH)^\top=0$, that is, by the vanishing of the tangential part
of the biharmonic equation. Hence it is weaker than the biharmonic condition.
Non-constant mean curvature biconservative surfaces in $\E^3$ are known, and complete examples have also been constructed \cite{CMOP,N}.
The same work of Ye\u{g}in \c{S}en and Turgay also shows that biconservative surfaces in $\E^4$ with parallel normalized mean curvature vector are certain rotational surfaces \cite{ST}.
Thus rotational surfaces arise naturally under the weaker biconservative condition.

Rotational symmetry plays an important role in geometric analysis. 
For example, rotationally symmetric models naturally appear in the study of self-similar solutions to geometric flows, such as Ricci solitons and Yamabe solitons \cite{Br,DS}. 
The harmonic map heat flow introduced by Eells and Sampson has played a fundamental role in the theory of harmonic maps \cite{ES}. 
As mentioned above, the study of biharmonic maps and submanifolds has also recently been connected with geometric flows through the biharmonic map flow. 
From this viewpoint, it is natural to investigate rotationally symmetric biharmonic submanifolds in higher codimension. 
Therefore, in this paper, we focus on simple rotational surfaces in $\E^4$ (cf. Section 48 in \cite{WM1916})
\[
X(s,\theta)=\bigl(x(s),y(s),r(s)\cos\theta,r(s)\sin\theta\bigr),
\]
where $r(s)>0$ and $s$ is the arc-length parameter of the profile curve
$\gamma(s)=\bigl(x(s),y(s),r(s)\bigr)\subset\E^3$.
In this paper, we prove that every biharmonic simple rotational surface in $\E^4$ is minimal, without imposing any additional assumptions.
\begin{theorem}\label{main}
Any biharmonic simple rotational surface in the four-dimensional Euclidean space is minimal.
\end{theorem}

This provides a partial affirmative answer to Chen's conjecture for higher codimension. 
In contrast to the above results, our theorem does not assume that the mean curvature vector satisfies any additional condition.

We now give an outline of the proof. First, the biharmonic equation is written as a system of ordinary differential equations for three functions $a,b,c$ determined by the profile curve. After introducing a new variable $\tau$ by $\tau=\int_{s_0}^s\frac{1}{r(w)}dw$, this system becomes
\[
\ddot a=0,
\qquad
\ddot b=0,
\qquad
\ddot c-c=0.
\]
Together with the arc-length condition, this yields a compatibility equation. We then classify all possible non-minimal branches into seven cases and show that each of them leads to a contradiction. The most general case is reduced to a finite system of polynomial equations, and a Groebner basis elimination gives the final obstruction.


\section{Rotational surfaces and the biharmonic equation}\label{sec2}

We consider a simple rotational surface in $\E^4$. Let
\[
X(s,\theta)=\bigl(x(s),y(s),r(s)\cos\theta,r(s)\sin\theta\bigr),
\]
where $r(s)>0$ and $s$ is the arc-length parameter of the profile curve
\[
\gamma(s)=\bigl(x(s),y(s),r(s)\bigr)\subset\E^3.
\]
Thus
\[
x'(s)^2+y'(s)^2+r'(s)^2=1.
\]
We have
\[
X_s=(x',y',r'\cos\theta,r'\sin\theta),
\qquad
X_\theta=(0,0,-r\sin\theta,r\cos\theta).
\]
Since $s$ is the arc-length parameter of the profile curve, the induced metric is
\[
g=ds^2+r(s)^2d\theta^2.
\]
Therefore, for an $\E^4$-valued function $Y$, the Laplace-Beltrami operator is given by
\[
\Delta Y
=\frac1r\frac{\partial}{\partial s}\left(r\frac{\partial Y}{\partial s}\right)
+\frac1{r^2}\frac{\partial^2Y}{\partial\theta^2}.
\]
A direct computation gives
\[
\Delta X
=
\left(
 x''+\frac{r'}r x',
 y''+\frac{r'}r y',
 \left(r''+\frac{r'^2-1}{r}\right)\cos\theta,
 \left(r''+\frac{r'^2-1}{r}\right)\sin\theta
\right).
\]
Set
\[
a:=x''+\frac{r'}r x',
\qquad
b:=y''+\frac{r'}r y',
\qquad
c:=r''+\frac{r'^2-1}{r}.
\]
Then the mean curvature vector is
\[
{\HH}=\frac12\Delta X
=\frac12\bigl(a,b,c\cos\theta,c\sin\theta\bigr).
\]
Moreover, for
\[
Y(s,\theta)=\bigl(f(s),g(s),h(s)\cos\theta,h(s)\sin\theta\bigr),
\]
one has
\[
\Delta Y=
\left(
 f''+\frac{r'}r f',
 g''+\frac{r'}r g',
 \left(h''+\frac{r'}r h'-\frac{h}{r^2}\right)\cos\theta,
 \left(h''+\frac{r'}r h'-\frac{h}{r^2}\right)\sin\theta
\right).
\]
Hence $\Delta\HH=0$ is equivalent to
\begin{equation}\label{eq-basic-system-s}
\begin{cases}
\displaystyle a''+\frac{r'}r a'=0,\\[0.8em]
\displaystyle b''+\frac{r'}r b'=0,\\[0.8em]
\displaystyle c''+\frac{r'}r c'-\frac{c}{r^2}=0.
\end{cases}
\end{equation}
On the other hand, ${\bf H}=0$ is equivalent to $a=b=c=0.$

We introduce a new variable $\tau$ by
\[
\tau=\int_{s_0}^s\frac{1}{r(w)} dw,~~(s_0\in \mathbb{R})
\]
and denote differentiation with respect to $\tau$ by a dot. Then
\[
f'=\frac{\dot f}{r}
\]
and hence
\[
f''+\frac{r'}r f'=\frac{\ddot f}{r^2}.
\]
Therefore, \eqref{eq-basic-system-s} becomes
\[
\ddot a=0,
\qquad
\ddot b=0,
\qquad
\ddot c-c=0.
\]
It follows that
\begin{equation*}
a(\tau)=A_1\tau+A_0,
\qquad
b(\tau)=B_1\tau+B_0,
\qquad
c(\tau)=C_+e^\tau+C_-e^{-\tau}.
\end{equation*}
Moreover,
\[
a=\frac{\ddot x}{r^2},
\qquad
b=\frac{\ddot y}{r^2},
\qquad
c=\frac{\ddot r-r}{r^2}.
\]
Thus
\begin{equation}\label{eq-xyr-tau}
\ddot x=r^2a,
\qquad
\ddot y=r^2b,
\qquad
\ddot r=r+r^2c.
\end{equation}
The arc-length condition becomes
\begin{equation}\label{eq-arc-tau}
\dot x^2+\dot y^2+\dot r^2=r^2.
\end{equation}
Differentiating \eqref{eq-arc-tau} and using \eqref{eq-xyr-tau}, we obtain the compatibility equation
\begin{equation}\label{eq-compatibility}
a\dot x+b\dot y+c\dot r=0.
\end{equation}

Set
\[
u(\tau):=(x(\tau),y(\tau)),
\qquad
q(\tau):=(a(\tau),b(\tau))=P\tau+Q,
\]
where
\[
P=(A_1,B_1),
\qquad
Q=(A_0,B_0).
\]


\section{Proof of Theorem \ref{main}}
We argue by contradiction to prove Theorem \ref{main}. Suppose that $\HH\neq0$. From the form obtained in the previous section, all possible non-minimal branches are divided into the following seven cases.
\begin{enumerate}
\item[{\rm (I)}]
$c\equiv0$ and $P= Q=0$.
\item[{\rm (II)}]
$c\equiv0$, $\det(P,Q)=0$ and $(P,Q)\not=(0,0)$.
\item[{\rm (III)}]
$c\equiv0$ and $\det(P,Q)\neq0$.
\item[{\rm (IV)}]
$c\neq0$ at some point, and $P= Q=0$.
\item[{\rm (V)}]
$c\neq0$ at some point, and $q\equiv {\rm const}\neq0$.
\item[{\rm (VI)}]
$c\neq0$ at some point, and $\det(P,Q)=0$ with $P\neq0$.
\item[{\rm (VII)}]
$c\neq0$ at some point, and $\det(P,Q)\neq0$.
\end{enumerate}

We rule out the cases (I)--(VII) in order.


\subsection{Exclusion of Case {\rm (I)}}\label{secA1}

\begin{proposition}\label{prop-caseA1}
Case {\rm (I)} cannot occur.
\end{proposition}

\begin{proof}
In this case, $a\equiv b\equiv c\equiv0$, that is, the surface is minimal, which is a contradiction.
\end{proof}


\subsection{Exclusion of Case {\rm (II)}}\label{secA23}

\begin{proposition}\label{prop-caseA23}
Case {\rm (II)} cannot occur.
\end{proposition}

\begin{proof}
Since $P,Q$ are parallel, we have $q(\tau)=\phi(\tau)e$ 
for some scalar function $\phi$ and some constant unit vector $e$, with $\phi\not\equiv0$.
By \eqref{eq-compatibility}, $\dot u \cdot q =0$. Hence, the vector $\dot u$ is perpendicular to $e$. On the other hand, \eqref{eq-xyr-tau} gives $\ddot u=r^2\phi e$. Taking the $e$-component and differentiating the orthogonality relation, we obtain $0=r^2\phi$, which is impossible. 
\end{proof}


\subsection{Exclusion of Case {\rm (III)}}\label{secA4}

\begin{proposition}\label{prop-caseA4}
Case {\rm (III)} cannot occur.
\end{proposition}

\begin{proof}
Assume that $c\equiv0$ and $\det(P,Q)\neq0$. Then \eqref{eq-xyr-tau} and \eqref{eq-compatibility} become
\[
\ddot u=r^2q,
\qquad
\ddot r=r,
\qquad
\dot u\cdot q=0.
\]
Since $\dot u \cdot q=0$ in $\R^2$, there exists a scalar function $\lambda(\tau)$ such that
\[
\dot u=\lambda Jq,
\]
where $J$ denotes the rotation by the angle $\frac{\pi}{2}$. Differentiating this equality gives
\[
\ddot u=\dot\lambda Jq+\lambda JP.
\]
On the other hand, $\ddot u=r^2q$. Taking the inner products with $q$ and $Jq$, respectively, we obtain
\[
\lambda\det(P,Q)=r^2 ||q|| ^2,
\qquad
\dot \lambda ||q||^2+\lambda ( P \cdot q )=0,
\]
where $||q||=\sqrt{q\cdot q}$.
From this, $\frac{d}{d\tau}\lambda||q||=0$. Hence, $\lambda||q||$ is constant, say $K$, that is 
\[
\lambda(\tau)=\frac{K}{||q(\tau)||}
\]
and hence
\[
r^2(\tau)=\frac{K\det(P,Q)}{||q(\tau)||^3}.
\]
After translating $\tau$ and rotating the $(x,y)$-plane, we may assume that
\[
P=(\alpha,0),
\qquad
Q=(0,\beta),
\qquad
\alpha\beta\neq0.
\]
Then
\[
q(\tau)=(\alpha\tau,\beta),
\qquad
||q(\tau)||^2=\alpha^2\tau^2+\beta^2,
\]
and therefore
\[
r=Lw^{-3/4},
\]
where 
\[
L:=\sqrt{K\det(P,Q)}(>0) \quad \text{and} \quad w(\tau):=\alpha^2\tau^2+\beta^2.
\]
Then
\[
\ddot r=Lw^{-11/4}\left(\frac{15}{4}\alpha^4\tau^2-\frac32\alpha^2\beta^2\right).
\]
Thus $\ddot r=r$ implies
\[
\frac{15}{4}\alpha^4\tau^2-\frac32\alpha^2\beta^2=(\alpha^2\tau^2+\beta^2)^2.
\]
The left-hand side is a polynomial of degree two in $\tau$, whereas the right-hand side is a polynomial of degree four. This is impossible as an identity. Hence Case {\rm (III)} cannot occur.
\end{proof}

\subsection{Exclusion of Case {\rm (IV)}}\label{secB1}

\begin{proposition}\label{prop-caseB1}
Case {\rm (IV)} cannot occur.
\end{proposition}

\begin{proof}
By \eqref{eq-compatibility}, we have $c\dot r=0$. Hence, on an open interval where $c\neq0$, we have $\dot r=0$. The third equation of \eqref{eq-xyr-tau} then gives $c=-1/r$, which is constant. However, the only constant solution of $\ddot c=c$ is $c=0$. This is a contradiction.
\end{proof}


\subsection{Exclusion of Case {\rm (V)}}\label{secB2}

\begin{proposition}\label{prop-caseB2}
Case {\rm (V)} cannot occur.
\end{proposition}

\begin{proof}
After a rotation in the $(x,y)$-plane, we may assume that
\[
q=(a,b)=(\alpha,0),
\qquad
\alpha>0.
\]
Then
\[
\ddot x=\alpha r^2,
\qquad
\ddot y=0,
\qquad
\ddot r=r+cr^2,
\qquad
\ddot c=c.
\]
Thus $\dot y=k$ is constant. The compatibility equation \eqref{eq-compatibility} becomes
\begin{equation}\label{eq-case1-comp}
\alpha \dot x+c\dot r=0.
\end{equation}
Combining this with \eqref{eq-arc-tau}, we get
\begin{equation*}
(\alpha^2+c^2)\dot r^2+\alpha^2k^2=\alpha^2r^2.
\end{equation*}
Differentiating \eqref{eq-case1-comp} and using $\ddot x=\alpha r^2$, we obtain
\begin{equation}\label{eq-case1-G}
\dot c\dot r+cr+(\alpha^2+c^2)r^2=0.
\end{equation}
Set
\[
G:=\dot c\dot r+cr+(\alpha^2+c^2)r^2.
\]
Then $G\equiv0$.

Since $\ddot c=c$, we have
\[
c(\tau)=C_+e^\tau+C_-e^{-\tau}
\]
and
$
\dot c^2-c^2
$
is constant. We divide the proof into three cases.

\smallskip
\noindent
{\bf Case 1.} $\dot c^2-c^2<0$.

In this case, $c$ has a critical point. After translating $\tau$, we may assume that
\[
\dot c(0)=0,
\qquad
c(0)=c_0\neq0.
\]
Put $A:=\alpha^2+c_0^2$. From \eqref{eq-case1-G} at $\tau=0$, we obtain
\[
c_0r(0)+Ar(0)^2=0.
\]
Since $r(0)>0$, this gives
\begin{equation*}
r(0)=-\frac{c_0}{A}.
\end{equation*}
Differentiating $G\equiv0$ twice and three times, and evaluating at $\tau=0$, we obtain
\[
\ddot G(0)=\frac{2A^3\dot r(0)^2+c_0^2(c_0^2-2\alpha^2)}{A^2},
\]
\[
\dddot G(0)=-\frac{10c_0^3}{A}\dot r(0),
\]
hence, $\dot r(0)=0$ and $c_0^2=2\alpha^2$. 
Therefore, we obtain
\[
G^{(4)}(0)=\frac{56}{9}>0,
\]
which is a contradiction.

\smallskip
\noindent
{\bf Case 2.} $\dot c^2-c^2>0$.

In this case, $c$ has a zero. After translating $\tau$, we may assume that
\[
c(0)=0,
\qquad
\dot c(0)=c_1\neq0.
\]
From \eqref{eq-case1-G} at $\tau=0$, we get
\begin{equation}\label{eq-case1-C2-1}
c_1\dot r(0)+\alpha^2r(0)^2=0.
\end{equation}
Moreover, using $\dot G(0)=0$, we obtain
\[
2r(0)(\alpha^2\dot r(0)+c_1)=0.
\]
Since $r(0)>0$, it follows that
\begin{equation}\label{eq-case1-C2-2}
\dot r(0)=-\frac{c_1}{\alpha^2}.
\end{equation}
Substituting \eqref{eq-case1-C2-2} into \eqref{eq-case1-C2-1}, we get
\[
r(0)^2=\frac{c_1^2}{\alpha^4}.
\]
Substituting this into $\ddot G(0)$, we have
\[
\ddot G(0)=\frac{3c_1^4}{\alpha^4}>0.
\]
This is a contradiction.

\smallskip
\noindent
{\bf Case 3.} $\dot c^2-c^2=0$.

Put $c(0)=c_0$.
In this case,
\[
c(\tau)=c_0e^{\sigma\tau},
\qquad
c_0\neq0,
\qquad
\sigma=\pm1.
\]
In particular, $\dot c=\sigma c$.  Solving \eqref{eq-case1-G} together with its first derivative at $\tau=0$, we obtain
\[
r(0)= -\frac{c_0(c_0^2+4\alpha^2)}{2(\alpha^2+c_0^2)^2},
\]
\[
\dot r(0)=
\frac{\sigma c_0(c_0^2-2\alpha^2)(c_0^2+4\alpha^2)}{4(\alpha^2+c_0^2)^3}.
\]
Substituting these into $\ddot G(0)=0$, we find
\[
\ddot G(0)=
\frac{c_0^4(c_0^2+4\alpha^2)^2(c_0^2+16\alpha^2)}{8(\alpha^2+c_0^2)^5}>0.
\]
This is again a contradiction.

Therefore Case {\rm (V)} cannot occur.
\end{proof}


\subsection{Exclusion of Case {\rm (VI)}}\label{secB3}

\begin{proposition}\label{prop-case2}
Case {\rm (VI)} cannot occur.
\end{proposition}

\begin{proof}
After translating $\tau$ and rotating the $(x,y)$-plane, we may assume that
\[
q(\tau)=(\alpha\tau,0),
\qquad
\alpha\neq0.
\]
Thus
\begin{equation}\label{eq-case2-ode}
\ddot x=\alpha\tau r^2,
\qquad
\ddot y=0,
\qquad
\ddot r=r+cr^2,
\qquad
\ddot c=c.
\end{equation}
Moreover, \eqref{eq-compatibility} and \eqref{eq-arc-tau} become
\begin{equation}\label{eq-case2-comp}
\alpha\tau\dot x+c\dot r=0,
\qquad
\dot x^2+\dot y^2+\dot r^2=r^2.
\end{equation}
Since $\ddot y=0$, $\dot y=k$ is constant.

Around $\tau=0$, write
\[
c(\tau)=c_0+c_1\tau+\frac{c_0}{2}\tau^2+\frac{c_1}{3!}\tau^3+\frac{c_0}{4!}\tau^4+\frac{c_1}{5!}\tau^5+\frac{c_0}{6!}\tau^6+\frac{c_1}{7!}\tau^7+\mathcal{O}(\tau^8),
\]
\[
r(\tau)=r_0+r_1\tau+\frac{r_2}{2}\tau^2+\frac{r_3}{3!}\tau^3+\frac{r_4}{4!}\tau^4+\frac{r_5}{5!}\tau^5+\frac{r_6}{6!}\tau^6+\frac{r_7}{7!}\tau^7+\mathcal{O}(\tau^8),
\qquad r_0>0.
\]

From this and the third equation of \eqref{eq-case2-ode}, 
we obtain the following formulas, which are also used in Case (VII).
\begin{claim}\label{Clm:r}
\begin{align}
 r_2&=r_0+c_0r_0^2,\label{eq-r2}\\
 r_3&=r_1+2c_0r_0r_1+c_1r_0^2,\label{eq-r3}\\
 r_4&=r_0+4c_0r_0^2+2c_0^2r_0^3+2c_0r_1^2+4c_1r_0r_1,\notag\\  
 r_5&=r_1+16c_0r_0r_1+10c_0^2r_0^2r_1+8c_0c_1r_0^3+8c_1r_0^2+6c_1r_1^2,\notag\\
 r_6&=r_0+25c_0r_0^2+34c_0^2r_0^3+10c_0^3r_0^4+22c_0r_1^2+20c_0^2r_0r_1^2\notag\\
 &\qquad \ +44c_1r_0r_1+56c_0c_1r_0^2r_1+8c_1^2r_0^3.\notag\\
 r_7&=r_1+138c_0r_0r_1+242c_0^2r_0^2r_1+80c_0^3r_0^3r_1+69c_1r_0^2\notag\\
 &\qquad \ +184c_0c_1r_0^3+86c_0^2c_1r_0^4+66c_1r_1^2+152c_0c_1r_0r_1^2\notag\\
 &\qquad \ +80c_1^2r_0^2r_1+20c_0^2r_1^3.\notag
\end{align}
\end{claim}

\smallskip
\noindent
{\bf Case 1.} $c_0=0$.

In this case, $c\not\equiv0$ implies $c_1\neq0$. Comparing the coefficient of $\tau^2$ in \eqref{eq-case2-comp}, we get
\[
r_2=0,
\]
because $c_1\not=0$.
On the other hand, by \eqref{eq-r2}, we have
\[
0=r_2=r_0.
\]
This contradicts $r_0>0$.

\smallskip
\noindent
{\bf Case 2.} $c_0\neq0$.

The constant term of \eqref{eq-case2-comp} gives
\[
c_0r_1=0.
\]
Hence $r_1=0$. 
Comparing the coefficient of $\tau^2$ in \eqref{eq-case2-comp}, we obtain
\[
c_1r_2+\frac{c_0}{2}r_3=0.
\]
Combining this with \eqref{eq-r2} and \eqref{eq-r3}, we obtain
\[
c_1r_0\left(1+\frac32 c_0r_0\right)=0.
\]
Therefore,
\[
c_1=0
\qquad\text{or}\qquad
c_0r_0=-\frac23.
\]

\smallskip
\noindent
{\bf Subcase 2a.} $c_1=0$.

In this case,
\[
c(\tau)=c_0+\frac{c_0}{2}\tau^2+\frac{c_0}{4!}\tau^4+\frac{c_0}{6!}\tau^6+\mathcal{O}(\tau^7),
\] Moreover, by Claim \ref{Clm:r},
\[
r_3=0,
\qquad
r_4=r_0(2t^2+4t+1),
\]
where $t:=c_0r_0$. Comparing the coefficients of $\tau^3$ and $\tau^5$ in \eqref{eq-case2-comp}, we get
\begin{equation}\label{p57}
3A+2t^3+7t^2+4t=0,
\end{equation}
\[
15A(t+1)+5t^4+27t^3+35t^2+8t=0,
\]
where $A:=\alpha^2r_0^2>0$. Eliminating $A$, we obtain
\[
5t^3+18t^2+20t+12=0.
\]
Set $f(t):=5t^3+18t^2+20t+12$. $f$ has exactly one zero. Since $f(-\frac{12}{5})<0$ and $f(-\frac{9}{4})>0$, we have $-\frac{12}{5}<t<-\frac{9}{4}.$ Set $\varphi(t):=2t^2+7t+4$. Since $\varphi'(t)<0$ on $-\frac{12}{5}<t<-\frac{9}{4}$ and $\varphi(-\frac{12}{5})<0$, we have $\varphi(t)<0$ on $-\frac{12}{5}<t<-\frac{9}{4}$. 
The equation \eqref{p57} can be written as
\[
3A=-t\varphi(t).
\]
The right-hand side is negative, whereas the left-hand side is positive. This is a contradiction.

\smallskip
\noindent
{\bf Subcase 2b.} $c_0r_0=-\frac23$.

By Claim \ref{Clm:r}, we have
\[
r_2=\frac{r_0}{3},
\qquad
r_3=c_1r_0^2,
\qquad
r_4=-\frac79r_0,
\qquad
r_5=\frac83c_1r_0^2.
\]
Comparing the coefficient of $\tau^3$ in \eqref{eq-case2-comp}, we get
\[
81r_0^2(\alpha^2+c_1^2)=4.
\]
Comparing the coefficient of $\tau^4$ in \eqref{eq-case2-comp}, we obtain
\[
-\frac{17}{54}c_1r_0=0.
\]
Thus $c_1=0$, which is the same condition as in Subcase 2a, hence we have a contradiction.
Therefore Case {\rm (VI)} cannot occur.
\end{proof}


\subsection{Exclusion of Case {\rm (VII)}}\label{secB4}

It remains to consider the generic case
\[
c\not\equiv0,
\qquad
\det(P,Q)\neq0.
\]
In this case, we use coefficient comparison and algebraic elimination. Although the computation is long, it follows only from a finite number of differentiations and Taylor expansions.

\begin{proposition}\label{prop-caseB4}
Case {\rm (VII)} cannot occur.
\end{proposition}

\begin{proof}
After translating $\tau$ and rotating the $(x,y)$-plane, we may assume that
\[
q(\tau)=(\alpha\tau,\beta),
\qquad
\alpha\beta\neq0.
\]
Thus
\begin{equation}\label{eq-case4-ode}
\ddot x=\alpha\tau r^2,
\qquad
\ddot y=\beta r^2,
\qquad
\ddot r=r+cr^2,
\qquad
\ddot c=c.
\end{equation}
By the same argument, Claim \ref{Clm:r} holds. 

Moreover, \eqref{eq-compatibility} and \eqref{eq-arc-tau} become
\begin{equation}\label{eq-case4-comp-arc}
\alpha\tau\dot x+\beta\dot y+c\dot r=0,
\qquad
\dot x^2+\dot y^2+\dot r^2=r^2.
\end{equation}
At $\tau=0$, put
\[
r(0)=r_0>0,
\qquad
\dot r(0)=r_1,
\qquad
c(0)=c_0,
\qquad
\dot c(0)=c_1.
\]

Set
\[
A:=\alpha^2r_0^2>0,
\qquad
B:=\beta^2r_0^2>0,
\qquad
p:=\frac{r_1}{r_0},
\qquad
t:=c_0r_0,
\qquad
u:=c_1r_0.
\]
Around $\tau=0$, write 
\[
x(\tau)=x_0+x_1\tau+\frac{x_2}{2}\tau^2+\frac{x_3}{3!}\tau^3+\frac{x_4}{4!}\tau^4+\frac{x_5}{5!}\tau^5+\frac{x_6}{6!}\tau^6+\mathcal{O}(\tau^7),
\]
\[
y(\tau)=y_0+y_1\tau+\frac{y_2}{2}\tau^2+\frac{y_3}{3!}\tau^3+\frac{y_4}{4!}\tau^4+\frac{y_5}{5!}\tau^5+\frac{y_6}{6!}\tau^6+\frac{y_7}{7!}\tau^7+\mathcal{O}(\tau^8).
\]
By \eqref{eq-case4-ode}, we have
\begin{align*}
x_2&=0,\\
x_3&=\alpha r_0^2,\\
x_4&=4\alpha r_0r_1,\\
x_5&=6\alpha r_1^2+6\alpha r_0^2+6\alpha c_0r_0^3,\\
x_6&=32\alpha r_0r_1+40\alpha c_0r_0^2r_1+8\alpha c_1r_0^3
\end{align*}
and
\begin{align*}
y_2&=\beta r_0^2,\\
y_3&=2\beta r_0r_1,\\
y_4&=2\beta r_1^2+2\beta r_0^2+2\beta c_0r_0^3,\\
y_5&=8\beta r_0r_1+10\beta c_0 r_0^2r_1+2\beta c_1r_0^3,\\
y_6&=8\beta r_0^2+20\beta c_0r_0^3+10\beta c_0^2r_0^4\\
&\quad+8\beta r_1^2+20\beta c_0r_0r_1^2+16\beta c_1r_0^2r_1,\\
y_7&=32\beta r_0r_1+132\beta c_0r_0^2r_1+36\beta c_1r_0^3\\
&\quad+80\beta c_0^2r_0^3r_1+36\beta c_0c_1r_0^4+20\beta c_0r_1^3+52\beta c_1r_0r_1^2.
\end{align*} 
By using \eqref{eq-case4-comp-arc} and eliminating the Taylor coefficients of $x$ and $y$, we obtain the following algebraic system:
\begin{align}
E_0&:=B(B+pu+t+t^2)^2+Ap^2(B+t^2)-AB=0,\label{eq:E0}\\
E_2&:=2Bp+2pt^2+2pt+3tu+2u=0,\label{eq:E2}\\
E_3&:=3A+2Bp^2+2Bt+2B+2p^2t^2+10ptu+4pu\label{eq:E3}\\
&\quad+2t^3+7t^2+4t+3u^2=0,\notag\\
E_4&:=8Ap+5Bpt+4Bp+Bu+7p^2tu+5pt^3+14pt^2+4pt\label{eq:E4}\\
&\qquad +8pu^2+8t^2u+17tu+4u=0,\notag\\
E_5&:=15Ap^2+15At+15A+10Bp^2t+4Bp^2+8Bpu\label{eq:E5}\\
&\qquad+5Bt^2+10Bt+4B+10p^2t^3+21p^2t^2\notag\\
&\qquad+15p^2u^2+53pt^2u+92ptu+8pu\notag\\
&\qquad +5t^4+27t^3+35t^2+24tu^2+8t+25u^2=0,\notag\\
E_6&:=120Apt+96Ap+24Au+10Bp^3t\notag\\
&\qquad+26Bp^2u+40Bpt^2+66Bpt+16Bp\notag\\
&\qquad +18Btu+18Bu+10p^3t^3+136p^2t^2u\notag\\
&\qquad +164p^2tu+40pt^4+196pt^3+204pt^2\notag\\
&\qquad+208ptu^2+16pt+172pu^2\notag\\
&\qquad+73t^3u+274t^2u+220tu+24u^3+16u=0\notag
\end{align}

Let 
\[
I:=\langle E_0,E_2,E_3,E_4,E_5,E_6\rangle \subset \mathbb{Q}[A,B,u,p,t].
\]
We calculate a Groebner basis of $I$. A Groebner basis computation with lexicographic order
$A \succ B \succ u \succ p \succ t$
shows that the reduced basis contains the polynomial
$pt^4$. 
Therefore, there exist $K_0,K_2,K_3,K_4,K_5,K_6 \in\mathbb{Q}[A,B,u,p,t]$ such that
\[
pt^4={K_0E_0+K_2E_2+K_3E_3+K_4E_4+K_5E_5+K_6E_6}.
\]
Since $E_0=E_2=E_3=E_4=E_5=E_6=0$, we have
\[
p=0
\qquad\text{or}\qquad
t=0.
\]
Here we remark that the explicit polynomials $K_i$ are omitted, since they are lengthy and
are not used elsewhere in the proof (They are available on the author's webpage).

\smallskip
\noindent
{\bf Case 1.} $t=0$ is impossible.

Suppose that $t=0$. Then \eqref{eq:E2} gives
\[
u=-Bp.
\]
If $p=0$, then \eqref{eq:E3} gives $3A+2B=0$, which contradicts $A>0$ and $B>0$. Hence $p\neq0$. Substituting $u=-Bp$ into \eqref{eq:E4}, we obtain
\[
8A+8B^2p^2-B^2=0,
\]
or equivalently
\[
A=B^2\left(\frac18-p^2\right).
\]
Thus $A>0$ implies $p^2<1/8$. On the other hand, substituting this into \eqref{eq:E3}, we obtain
\[
0=\frac38B^2+2B(1-p^2).
\]
The right-hand side is positive. This is a contradiction. Therefore $t=0$ is impossible.

\smallskip
\noindent
{\bf Case 2.} $p=0$ is also impossible.

Suppose that $p=0$. Then \eqref{eq:E2} gives
\[
(3t+2)u=0.
\]

\smallskip
\noindent
{\bf Case 2a.} $u=0$.

In this case, \eqref{eq:E0} gives
\[
A=(B+t+t^2)^2.
\]
Furthermore, \eqref{eq:E3} and \eqref{eq:E5} become
\[
3(B+t+t^2)^2+2B(t+1)+2t^3+7t^2+4t=0,
\]
\[
15(B+t+t^2)^2(t+1)+5Bt^2+10Bt+4B+5t^4+27t^3+35t^2+8t=0.
\]
Combining these equations, we obtain 
\[
B=-\frac{5t^4+18t^3+20t^2+12t}{5t^2+10t+6}
\]
and eliminating $B$ from the above two equations, we get
\[
12t^2(18t^4+73t^3+140t^2+120t+54)=0.
\]
However,
\[
18t^4+73t^3+140t^2+120t+54
=18\left(t^2+\frac{73}{36}t+\frac{120}{73}\right)^2
+\frac{2612159}{383688}t^2+\frac{2056752}{383688}>0.
\]
Thus $t=0$. This contradicts Case 1.

\smallskip
\noindent
{\bf Case 2b.} $3t+2=0$ and $u\neq0$.

Substituting $p=0$ and $t=-\frac{2}{3}$ into \eqref{eq:E4}, we obtain
\[
B=\frac{34}{9}.
\]
Moreover, \eqref{eq:E0} gives
\[
A=(B+t+t^2)^2=\left(\frac{32}{9}\right)^2.
\]
Substituting these into \eqref{eq:E3}, we obtain
\[
0=3A+\frac23B-\frac{4}{27}+3u^2
=\frac{1088}{27}+3u^2>0.
\]
This is a contradiction.

Therefore $p=0$ is also impossible. 

This contradiction shows that Case {\rm (VII)} cannot occur.
\end{proof}



\subsection{Final argument}\label{secmain}

\begin{proof}[Proof of Theorem~\ref{main}]
If ${\bf H}\neq0$, by Propositions~\ref{prop-caseA1}--\ref{prop-caseB4}, all possible non-minimal branches {\rm (I)}--{\rm (VII)} are excluded. Hence any simple rotational surface satisfying $\Delta\HH=0$ must satisfy $\HH=0.$
This proves that any biharmonic simple rotational surface in $\E^4$ is minimal.
\end{proof}

\noindent
{\bf Acknowledgements.}~\\
The author would like to express his gratitude to Jun-ichi Inoguchi and Cezar Oniciuc for their useful comments.
~\\

\section*{Statements and Declarations}
There is no conflict of interest in the manuscript.

\section*{Data availability statement} 
Data sharing not applicable to this article
as no datasets were generated or analysed during the current study.

\section*{Supplementary material}
Supplementary material related to the finite algebraic elimination is available at the author's webpage: 
\url{https://sites.google.com/site/shunmaetahomepage/supplementary-materials}

\bibliographystyle{amsbook}

\end{document}